\documentclass[11pt]{amsart}
\usepackage{amsmath}
\usepackage{enumerate}
\newtheorem{lemma}{Lemma}[section]
\newtheorem{otheorem}[lemma]{Theorem}
\newtheorem{theorem}{Theorem}

\newtheorem{proposition}[lemma]{Proposition}
\theoremstyle{definition}
\newtheorem{Def}[lemma]{Definition}

\newtheorem{question}[lemma]{Question}
\newcommand{\pl}{\mathcal{PML}}
\newcommand{\ml}{\mathcal{ML}}
\newcommand{\Int}{\operatorname{Int}}

\let\blackboard\mathbb
\newcommand{\complexes}{\blackboard{C}}
\newcommand{\hyperbolic}{\blackboard{H}}

\newcommand{\reals}{\blackboard{R}}
\newcommand{\naturals}{\blackboard{N}}
\newcommand{\rationals}{\blackboard{Q}}
\newcommand{\mcg}{\operatorname{MCG}}
\newcommand{\length}{\operatorname{length}}
\newcommand{\psl}{\mathrm{PSL}_2 \complexes}
\newcommand{\cc}{\mathcal{CC}}
\newcommand{\ie}{i.e.\ }
\newcommand{\aut}{\operatorname{Aut}}
\newcommand{\maptripod}{\tau}
\newcommand{\Fr}{\operatorname{Fr}}

\begin{document}
\title[Heegaard splittings]{Subgroups of mapping class groups related to Heegaard splittings and bridge decompositions}
\author{Ken'ichi Ohshika}
\address{Department of Mathematics\\
Graduate School of Science\\
Osaka University\\
Toyonaka, Osaka, 560-0043 , Japan}
\email{ohshika@math.sci.osaka-u.ac.jp}

\author{Makoto Sakuma}
\address{Department of Mathematics\\
Graduate School of Science\\
Hiroshima University\\
Higashi-Hiroshima, 739-8526, Japan}
\email{sakuma@math.sci.hiroshima-u.ac.jp}

\subjclass[2010]{Primary 57M50, 57M07, 30F40, 20F34\\
\indent {The first author was supported 
by JSPS Grants-in-Aid 22654008.
The second author was supported
by JSPS Grants-in-Aid 21654011.}}

\begin{abstract}
Let $M=H_1\cup_S H_2$ be a Heegaard 
splitting
of a closed orientable $3$-manifold $M$ (or a bridge decomposition of a link exterior).
Consider the subgroup $\mcg^0(H_j)$ of 
the mapping class group of $H_j$ consisting of 
mapping classes represented by auto-homeomorphisms of $H_j$ homotopic to the identity, and let
$G_j$ be the subgroup of the automorphism group of the curve complex
$\cc(S)$ obtained as the image of $\mcg^0(H_j)$.
Then the group $G=\langle G_1, G_2\rangle$ generated by 
$G_1$ and $G_2$ preserve the homotopy class in $M$ of 
simple loops on $S$.
In this paper, 
we study the structure of the group $G$ and
the problem to what extent the converse to this observation holds. 
\end{abstract}
\date{}
\maketitle
Let $M$ be a closed orientable $3$-manifold
and $S$ a Heegaard surface of $M$.
Then $M$ is decomposed into two handlebodies $H_1$ and $H_2$
such that $S=\partial H_1=\partial H_2$.
We consider   the (extended) mapping class group of $S$,
i.e., the group of isotopy classes of (possibly orientation-reversing)
auto-homeomorphisms of $S$, and denote it by $\mcg(S)$. 
Let $\mcg(H_j)$ denote the mapping class group of $H_j$ ($j=1,2$).
Then $\mcg(H_j)$ 
can be identified with a subgroup of $\mcg(S)$,
by restricting an auto-homeomorphism of $H_j$ to $S$.
We consider the subgroup of $\mcg(H_j)$ consisting of mapping classes represented by auto-homeomorphisms of $H_j$ homotopic to the identity, and denote it by
 $\mcg^0(H_j)$.

Now, let $\cc(S)$
be the curve complex of $S$,
namely the simplicial complex 
each of whose vertex represents an isotopy class of 
essential simple closed curves in $S$
and each of whose simplex represents a set of isotopy classes with pairwise disjoint representatives. 
Then 
it is known
that
the natural action of $\mcg(S)$ on $\cc(S)$ induces a surjection
from $\mcg(S)$ onto the simplicial automorphism group $\aut(\cc(S))$
whose kernel is trivial or the cyclic group of order $2$ 
generated by hyper-elliptic involution
depending on whether the genus of $S$ is greater than or equal to $2$
(see \cite[Section 8]{Ivanov}, \cite{Kor} and \cite{Luo}).
In this way, 
$\mcg(S)$ (or its quotient by the order $2$ cyclic group)
is canonically identified with $\aut(\cc(S))$.

Let $G_j$ be the image of the subgroup  $\mcg^0(H_j)$ of $\mcg(S)$
in $\aut(\cc(S))$, and let $G=\langle G_1,G_2\rangle$
be the subgroup of $\aut(\cc(S))$ generated by $G_1$ and $G_2$.
Then $G$ preserves the homotopy classes  in $M$ of the simple 
closed curves on $S$,
namely, for any $g\in G$ and for any vertex $\alpha$ of $\cc(S)$,
$g\alpha$ is homotopic to $\alpha$ in $M$.
(We 
ignore the distinction between 
a vertex of $\cc(S)$ and a simple 
closed curve
in $S$ representing the vertex.)

Let $\Delta_j$ be the subset of the vertex set of $\cc(S)$
consisting of the meridians of $H_j$,
namely the set of vertices in $\cc(S)$
represented by simple closed curves 
which bound discs in $H_j$.
Let $Z$ be the set of vertices in $\cc(S)$
represented by simple closed curves 
which are null-homotopic in $M$.
Then, by the above observation, the orbit $G(\Delta_1\cup\Delta_2)$
is contained in $Z$.
The following natural question was posed
by Minsky in \cite[Question ~5.4]{Gordon}.

\begin{question}
\label{Minsky's question}
When is $Z$ equal to the orbit $G(\Delta_1\cup\Delta_2)$?
\end{question}

The same question makes sense not only for Heegaard surfaces
but also for bridge spheres as follows.
Let $K$ be a knot or a link in $S^3$,
and let $S$ be a bridge sphere of $K$. 
Then $(S^3,K)$ is a union of two trivial tangles $(B^3_1,t_1)$
and $(B^3_2,t_2)$ such that
$(S,S\cap K)=\partial (B^3_1,t_1)=\partial (B^3_2,t_2)$.
Here a trivial tangle means
 a pair of a $3$-ball $B^3$ and 
mutually disjoint arcs properly embedded in $B^3$ 
which are simultaneously parallel to mutually disjoint arcs in $\partial B^3$.
We denote the punctured sphere $S-K$ by the same symbol $S$,
and consider the (extended) mapping class group $\mcg(S)$ of 
the punctured sphere $S$. 
Then the mapping class group 
$\mcg(B^3_j,t_j)$ 
of the pair $(B^3_j,t_j)$
can be
 identified with a subgroup of $\mcg(S)$,
by restricting an auto-homeomorphism
of $(B^3_j,t_j)$ to $S$.
Consider the subgroup of 
$\mcg(B^3_j,t_j)$ 
consisting of mapping classes represented by homeomorphisms pairwise-homotopic to the identity,
 and denote it by 
 $\mcg^0(B^3_j,t_j)$.
Let $G_j$ be the image of the subgroup  
 $\mcg^0(B^3_j,t_j)$.
of $\mcg(S)$
in $\aut(\cc(S))$, and let $G=\langle G_1,G_2\rangle$
be the subgroup of $\aut(\cc(S))$ generated by $G_1$ and $G_2$.
Let $\Delta_j$ be the set of vertices in $\cc(S)$
represented by simple closed curves 
which bound discs in $B^3_j-t_j$, and
let $Z$ be the set of vertices in $\cc(S)$
represented by simple closed curves 
which are null-homotopic in 
the link complement $M:=S^3-K$.
Then we can see that
the orbit $G(\Delta_1\cup\Delta_2)$ is contained in $Z$.
This observation was a starting point of \cite{Ohtsuki-Riley-Sakuma},
which gave rise to
a systematic construction of
epimorphisms between 2-bridge link groups.
Again, it is natural to ask when
$Z$ is equal to $G(\Delta_1\cup\Delta_2)$ 
(cf. \cite[Question 9.1]{Ohtsuki-Riley-Sakuma}).

In the second author's joint work with Donghi Lee \cite{lee_sakuma_1},
a complete answer to the above question for $2$-bridge links was given.
Moreover, the following results were obtained in 
a series of joint work 
\cite{lee_sakuma_0,lee_sakuma_1,lee_sakuma_2},
and they were applied in \cite{lee_sakuma_3} to give a variation of McShane's identity for $2$-bridge links.

\begin{theorem}
\label{thm_lee-sakuma}
Let $K$ be a $2$-bridge link in $S^3$ which is neither the trivial knot
nor the $2$-component trivial link,
and let $S$ be a $2$-bridge sphere of $K$.
Let $G=\langle G_1,G_2\rangle$, $\Delta_j$, and $Z$ be as explained above, and
let $\Lambda(G)$ and $\Omega(G)=\pl(S)-\Lambda(G)$, respectively, 
be the limit set and the domain of discontinuity of the action $G$ on $\pl(S)$.
Then the following hold.
\begin{enumerate}[\indent \rm (1)]
\item
The set $Z$ is equal to the orbit $G(\Delta_1\cup\Delta_2)$.
\item
The closure of $Z=G(\Delta_1\cup\Delta_2)$ in $\pl(S)$
is equal to the limit set $\Lambda(G)$.
Moreover $\bar Z=\Lambda(G)$ has measure $0$ in $\pl(S)$.
\item
The domain of discontinuity $\Omega(G)$ has full measure in $\pl(S)$,
and no essential simple closed curve in $S$
representing a point in $\Omega(G)$ is null-homotopic in $M=S^3-K$.
\item
Suppose that $K$ is neither 
a torus link
nor a twist knot.
Then no essential simple closed curve in $S$
representing a point in $\Omega(G)$ is peripheral in $M$,
i.e., no such simple closed curve 
is homotopic to a closed curve 
in a peripheral torus $\partial N(K)$ in $M$.
\item
Suppose that $K$ is neither 
a torus link
nor the Whitehead link.
Then, for any two essential simple 
closed curves in $S$
representing distinct points in $\Omega(G)$,
they are homotopic in $M$
if and only if they lie in the same $G$-orbit.
\item
The group $G$ is isomorphic to the free product $G_1*G_2$.
\end{enumerate}
\end{theorem}

It is natural to ask if the above theorem holds in 
a more general setting (see \cite{Sakuma}).
The purpose of this paper is to give 
the following partial answers to this natural question.
\begin{enumerate}
\item
If $S$ is a Heegaard surface or a bridge sphere with
sufficiently high Hempel distance, then the subgroup $G=\langle G_1,G_2\rangle$ of $\aut(\cc(S))$ 
 is isomorphic to the free product $G_1*G_2$
 (Theorem \ref{free product}).
 \item
If $S$ is a Heegaard surface, with $R$-bounded combinatorics for some $R>0$,
of a closed orientable hyperbolic $3$-manifold $M$ 
and if the Hempel distance of $S$ is larger than a constant $K_0$,
depending only on the topological type of $S$ and the constant $R$,
 then there is a non-empty open set $O$ in the projective measured lamination space
 $\pl(S)$, such that
 \begin{enumerate}
 \item
no simple closed curve
in $S$ representing a point in $O$ 
is null-homotopic in $M$,
\item
two simple closed curves 
in $S$ 
representing distinct points in $O$
cannot be homotopic in $M$.
 \end{enumerate}
 In particular, the action of $G$ on $\pl(S)$ has a non-empty domain
 of discontinuity (Theorems \ref{main1} and \ref{discontinuous open set}).
 \item
 Suppose that $M_n$ is obtained from two handlebodies by 
 an $n$-time iteration of a generic pseudo-Anosov map 
 $\phi$
 and consider the Heegaard splitting of $M_n$
 consisting of the two handlebodies.
 Then the subset $O$ in Theorem \ref{main1} 
 for the Heegaard surface of $M_n$
 can be made almost cover the entire projective lamination space
 so that the almost every point in the projective lamination space is contained in the open subset $O$ for $M_n$ with sufficiently large $n$ (Theorem \ref{iteration}).
\end{enumerate}

We note that 
the result of Namazi \cite{NaT} implies that
if the Hempel distance of a Heegaard splitting
$M=H_1\cup H_2$ is sufficiently large, then
$G_1 \cap G_2$ is finite.
More generally, the result of 
Johnson \cite{Johnson} implies that $G_1 \cap G_2$ is finite 
if the Hempel distance is greater than $3$.
Thus Theorem \ref{free product} may be regarded as a partial refinement
of these consequences of the results of  \cite{NaT} and \cite{Johnson}.
We also note that 
Theorems \ref{main1} and \ref{iteration}
may be regarded as 
a variant of the asymptotical faithfulness of the homomorphism
$\pi_1(H_i)\to \pi_1(M)$ established by Namazi \cite[Theorem 1.6]{Na}
and Namazi and Souto \cite[Lemma 6.1 and Theorem 6.1]{NS}.

The authors would like to thank Brian Bowditch
for his essential contribution to the proof of Theorem \ref{free product}, without which they should not have been able to complete the work.
They would also like to thank Jeff Brock and Yair Minsky
for stimulating conversation, valuable comments on the first version of this paper, and allowing them to read a draft of their joint work \cite{BMNS}
with Hossein Namazi and Juan Souto on model manifolds,
on which Theorems \ref{main1} and \ref{discontinuous open set} depend.

\section{Structure of the group $\langle G_1, G_2 \rangle$}

In this section, we shall prove the following theorem.

\begin{theorem}[Bowditch-Ohshika-Sakuma]
\label{free product}
There is a constant $K_0$ depending only on the topological type of $S$
with the following property.
For a Heegaard splitting or a bridge decomposition
$M=H_1 \cup_S H_2$ with its Hempel distance greater than $K_0$, 
the group $\langle G_1, G_2 \rangle$ is decomposed into a free product $G_1 * G_2$,
where $G_1$ and $G_2$ are subgroups of $\aut(\cc(S))$ defined in the introduction.
\end{theorem}

We recall the terminology in Gromov's theory of hyperbolic metric spaces.
Let $X$ be a geodesic space, \ie  a metric space in which every pair of points can be connected by a geodesic segment.
Let $\triangle=\overline{PQ} \cup \overline{QR} \cup \overline{RP}$ be a geodesic triangle  with its vertices $P,Q$ and $R$ in $X$.
We consider a map, 
$\maptripod_{\triangle}$,
from $\triangle$ to a \lq\lq tripod" $T_\triangle$ which is 
an edge-wise isometry. We call this map the 
{\it comparison map} to $T_\triangle$.
The map $\maptripod_{\triangle}$ has a property that two points $a \in \overline{PQ}$ and $b \in \overline{QR}$ are identified under $\maptripod_{\triangle}$ if and only if 
$d(Q,a)=d(Q,b) \leq (P|R)_Q$, where the last term is the Gromov product defined by
$(P|R)_Q=\frac{1}{2}\left(d(P,Q)+d(R,Q)-d(P,R)\right)$.
The same holds even if we permute $P,Q$ and $R$.
Now, the triangle $\triangle$ is said to be 
{\it $\delta$-thin} when for each pair of points $a, b \in \triangle$ with $\maptripod_{\triangle}(a)=\maptripod_{\triangle}(b)$, we have $d_X(a,b) \leq \delta$.
A geodesic space $X$ is said to be {\it $\delta$-hyperbolic}
if every triangle in $X$ is $\delta$-thin.

In the following argument, we shall use the Gromov hyperbolicity of the curve complex $\cc(S)$
and the quasi-convexity of $\Delta_1$ and $\Delta_2$,
where $\Delta_1$ and $\Delta_2$ are the subcomplexes of $\cc(S)$
spanned by the simple 
closed curves
bounding disks in $H_1$ and $H_2$, respectively.
(See Masur-Minsky \cite{MaMi0} and \cite{MaMi}.)
Let $\delta$ be a positive constant such that $\cc(S)$ is $\delta$-hyperbolic.
Let $L$ be a constant depending only on the topological type of $S$  such that both $\Delta_1$ and $\Delta_2$ are $L$-quasi-convex: any geodesic segment connecting two points in $\Delta_1$ (resp. $\Delta_2$) lies in the $L$-neighbourhood of $\Delta_1$ (resp. $\Delta_2$).

Now, we start to prove that $\langle G_1, G_2\rangle$ is decomposed as $G_1*G_2$.
Take a shortest geodesic segment $\gamma$ connecting $\Delta_1$ and $\Delta_2$ in $\cc(S)$, and denotes its endpoint in $\Delta_1$ by $x_\gamma$ and the one in $\Delta_2$ by $z_\gamma$.
Consider 
a word
$g_1 h_1 \dots  g_p h_p$, where $g_j \in G_1, h_j \in G_2$ and suppose that none of them is the identity.
We show that the element of $\aut(\cc(S))$ determined by 
the word $g_1 h_1 \dots  g_p h_p$ is nontrivial.
In general, we need to consider the case where $g_1=1$ or $h_p=1$, but the argument needs no modification even in these cases.

Consider the translates $g_1 \gamma$  of $\gamma$.
Then the endpoint $x_{\gamma}$ of $\gamma$ and the endpoint $g_1x_{\gamma}$ of $g_1\gamma$ are  both contained in $\Delta_1=g_1 \Delta_1$.
We connect $x_{\gamma}$ and $g_1x_{\gamma}$ by a geodesic segment $\delta_1$,
which lies in the $L$-neighbourhood of $\Delta_1$.
Next we consider the translate $g_1h_1\gamma$, and connect the endpoint $g_1z_{\gamma}$ of
$g_1\gamma$
with the endpoint $g_1h_1z_{\gamma}$ of
$g_1h_1\gamma$
by a geodesic segment $\delta_1'$, which lies in the $L$-neighbourhood of 
$g_1\Delta_2=g_1h_1\Delta_2$.
Repeating this process, we construct a 
piecewise geodesic arc
$\alpha=\gamma \cup \delta_1 \cup g_1 \gamma \cup \delta_1' \cup g_1h_1\gamma \cup \delta_2 \cup g_1h_1g_2\gamma \cup \delta_2' \cup\dots \cup \delta_p' \cup g_1 h_1\dots g_ph_p \gamma$.

Let $d$ be either $\delta_j$ or $\delta_j'$ in $\alpha$, and
let $c$ be its preceding geodesic segment, that is, $g_1h_1 \dots h_{j-1} \gamma$ for $\delta_j$ or 
$g_1h_1 \dots g_jh_j\gamma$ for $\delta_j'$. 
We connect the endpoints of $c\cup d$ by a geodesic segment, and denote it by $e$.
Let $\triangle$ be the geodesic triangle $c \cup d \cup e$, and $\maptripod_{\triangle}: \triangle \rightarrow T_\triangle$ the comparison map to a tripod as explained above.
Then the $L$-quasi-convexity of $\Delta_1, \Delta_2$ implies the following lemma.

\begin{lemma}
\label{geodesic in Delta}
There is a constant $L'$ 
depending only on $\delta$ and $L$
 (hence only on the topological type of $S$) 
 such that 
the longest subsegment of $d$ that is identified with
a subsegment of $c$
under the map $\maptripod_{\triangle}$ has length at most $L'$.
\end{lemma} 
\begin{proof}
Set $x=c \cap d$, $y=d \cap e$ and $z=c\cap e$.
By translating the entire picture so that $c$ becomes $\gamma$, we have only to consider the case where either $x, y$ lie in $\Delta_1$ and $z$ lies in $\Delta_2$ or $x,y$ lie in $\Delta_2$ and $z$ lies in $\Delta_1$.
We may assume that $x$ and $y$ lie in $\Delta_1$, because
we can argue in the same way also in the case where they lie in $\Delta_2$.
Let $\ell$ be the length of 
the longest subsegment of $d$ which is identified with
a subsegment of $c$
under $\maptripod_{\triangle}$, and let $p$ be its endpoints other than $x$.
Then $d(x,p)=\ell$ and there is a point $q$ on $c$ such that $\maptripod_{\triangle}(p)=\maptripod_{\triangle}(q)$ and hence $d(p,q) \leq \delta$ by the $\delta$-thinness of the triangle $\triangle=c\cup d\cup e$.
Since $q$ lies in the shortest geodesic segment $\gamma$ connecting $\Delta_1$ and $\Delta_2$,
we have $d(q,x)=d(q,\Delta_1)$.
Hence $\ell=d(q,x)=d(q,\Delta_1)\le d(q,p)+d(p,\Delta_1)\le \delta+L$.
Thus, by setting $L'$ to be $L+\delta$, we are done.
\end{proof}

Next, let $f$ be the geodesic segment in $\alpha$ 
following $d$, and 
$h$
a geodesic segment connecting $e \cap c$ and the endpoint of $f$ other than $d \cap f$.
We consider the geodesic triangle 
$e \cup f \cup h$, 
which we denote by $\triangle'$, and the comparison map to its corresponding  tripod $\maptripod_{\triangle'}: \triangle' \rightarrow T_{\triangle'}$.

\begin{lemma}
\label{geodesic in Delta'}
There is a constant $L''$ depending only on 
the topological type of
$S$ such that 
the longest subsegment of $e$ identified with a subsegment of $f$ 
under the comparison map $\maptripod_{\triangle'}$ has length at most $L''$.
\end{lemma}

To prove this lemma, we shall use the following 
consequence of Bowditch's theorem established in \cite{Bo} on
acylindricity of the mapping class group action on the curve complex.
The proof is deferred to the following section.

\begin{proposition}
\label{only identity}
For any $E>0$, there exists $L_0>0$ depending only on $E$ and 
the topological type of
$S$ for which the following holds.
Suppose that  $\gamma$  contains a subsegment $\gamma'$ with length at least $L_0$ and $g \in G_i$ such that $d(z, gz) \leq E$ for every $z \in \gamma'$.
Then $g$ is the identity.
\end{proposition}

\begin{proof}[Proof of Lemma \ref{geodesic in Delta'}
 assuming Proposition \ref{only identity}]
As was done in the proof of Lemma \ref{geodesic in Delta}, we can assume that  
$f$ is  
$g_1\gamma$
and that 
$h$
connects the endpoint of $\gamma$ in $\Delta_2$ 
and that of $f$ in $g_1 \Delta_2$.
Let $f'$ denote  the longest subsegment of $f$ starting from $y=d\cap f$ that is identified with a subsegment 
of
$e$ under the comparison map $\maptripod_{\triangle'}$.
We may assume $\length(f') = L+2\delta+K$ for some $K>0$,
for otherwise the assertion of the lemma obviously holds.

First suppose that $\length(d) > L+L'+2\delta$
for $L'$ in Lemma \ref{geodesic in Delta}.
Let $w$ be a point on $f'$ such that $d(y, w)$ is slightly bigger than $L+2\delta$, which is guaranteed to exist since $\length(f')=L+2\delta+K > L+2\delta$.
Then there is a point $w'$ lying on $e$ with $\maptripod_{\triangle'}(w)=\maptripod_{\triangle'}(w')$ and hence $d(w,w') \leq \delta$.
By Lemma \ref{geodesic in Delta}, the longest subsegment of $d$ starting from $y$ that is identified with 
a subsegment of $e$
under $\maptripod_{\triangle}$ has length greater than 
$\length(d)-L'$,
which in turn is greater than $(L+L'+2\delta)-L'=L+2\delta$ by assumption.
Thus we may assume, by choosing $w$ so that  $L+2\delta<d(y,w)< \length(d)-L'$,
that there is a point $w''$ on $d$ with $\maptripod_{\triangle}(w')=\maptripod_{\triangle}(w'')$, $d(w',w'') \leq \delta$, 
and hence $d(w,w'') \leq 2\delta$.
By the $L$-quasi-convexity of $\Delta_i$, the distance from $w''$ to 
$\Delta_i$
is at most $L$.
Therefore, we can connect $w$ by an arc of length at most $L+2\delta$ to $\Delta_1$, which we denote by $\zeta$.
The length of  $f \setminus f'|_{[y,w]} \cup \zeta$ is less than $\length(f)$, where $f'|_{[y,w]}$ denotes the subsegment of $f'$ between $y$ and $w$.
This contradicts the fact that $f$ is the shortest geodesic segment connecting $g_1 \Delta_2$ to $\Delta_1$.

Next suppose that $\length(d)\leq L+L'+2\delta$.
Let $f''$ be the longest subsegment of $f$  identified with a subsegment of
$e$ under $\maptripod_{\triangle'}$ and then with that of $c$ under $\maptripod_{\triangle}$.
Then $\length(f'')>\length(f')-\length(d)\geq (L+2\delta+K)-(L+L'+2\delta)\geq K-L'$.
Let $c''$ be the subsegment of $c$ identified with $f''$ as in the above,
and let $\varphi:c''\to f''$ be the isometry arising from the identification.
Let $x''$ be the endpoint of $c''$ nearer to $x$.
Then $d(x,x'')\leq L'$ by Lemma \ref{geodesic in Delta}, and hence $d(g_1(x),g_1(x'')) \leq L'$.
Thus $g_1(x'')$ is contained in the component of $f\setminus f''$ containing $y=g_1(x)$.
Since the length of the component of $f\setminus f''$ is less than that of $d$, 
we have $d(g_1(x''), \varphi(x''))< \length(d)\leq L+L'+2\delta$.
Since both $g_1$ and $\varphi$ are isometries
into $f$,
this implies that $d(g_1(\xi), \varphi(\xi))< L+L'+2\delta$
for every $\xi\in c''$.
Hence we have
$d(\xi,g_1(\xi))\leq d(\xi,\varphi(\xi))+d(\varphi(\xi),g_1(\xi))<2\delta+ L+L'+2\delta$
for every $\xi\in c''$.
Therefore, we see by 
Proposition \ref{only identity} 
that
there is a constant $L_0$ depending only on $S$ which bounds 
$\length(c'')=\length(f'')>K-L'$
from above.
Thus, $K$ is bounded by a constant $K'$ depending only on $S$.
By setting $L''$ to be $L+2\delta+K'$, we are done.
\end{proof}

\begin{proposition}
\label{quasi-geodesic}
There are constants $A,B$ depending only on $L$ and $\delta$ such that $\alpha$ is an $(A,B)$-quasi-geodesic if $K_0$ is large enough.
\end{proposition}

\begin{proof}[Proof of Proposition \ref{quasi-geodesic}]
We shall first show that there are $A,B$ as above such that $c\cup d \cup f$ is an $(A,B)$-quasi-geodesic.
By Lemma \ref{geodesic in Delta}, except for geodesic segments starting from $x$ of lengths at most $L'$, one on $c$ and the other on $d$, each point of  $c \cup d$ is within the distance $\delta$ from $e$.
In the same way, we see that by Lemma \ref{geodesic in Delta'}, except for geodesic segments starting from $y$ of length $L''$, one on $e$ and the other on $f$, every point on $e \cup f$ is within the distance $\delta$ from 
$h$.
These imply that $c \cup d \cup f$ is a $(1, 2L'+2L''+4\delta)$-quasi-geodesic.

This holds for every three consecutive arcs constituting $\alpha$ that has a translate of $\delta$ or $\delta'$ in the middle.
Now, in general in a $\delta$-hyperbolic geodesic space, for any $(C,D)$ there are $(A,B)$ and $l$ such that an arc each of whose subarc of length less than $l$ is $(C,D)$-quasi-geodesic is always $(A,B)$-quasi-geodesic itself, where $A,B$ and $l$ depend only on $C,D$ and $\delta$.
(See \cite[Chap.\ 3, Th\'{e}or\`eme 1.4]{CDP} or 
\cite[Chap.\ III, Theorem 1.13]{BH}.)
Therefore, by taking $K_0$, which bounds the length of $\gamma$ from below, to be large enough, we see that there are $A,B$ depending only on $\delta$ and $L$ such that $\alpha$ is an $(A,B)$-quasi-geodesic.
\end{proof}

Let $\alpha'$ be the subarc of $\alpha$ obtained by deleting the last geodesic segment $g_1h_1 \dots g_ph_p\gamma$ from $\alpha$.
Proposition \ref{quasi-geodesic} implies that there is a constant $C$ depending only on $A$ and $B$ such that  the endpoints of $\alpha'$, which are $x_\gamma$ and  $g_1h_1 \dots g_ph_p x_\gamma$, cannot be the same if 
$\length(\gamma) \geq C$.
Therefore, assuming $K_0$ to be greater than $C$, we see that the word
$g_1h_1 \dots g_ph_p$ represents a non-identity element.
As is noticed at the beginning of the proof,
the same conclusion holds when $g_1=1$ or $h_p=1$.
Hence $G=G_1*G_2$.
This completes the proof of Theorem \ref{free product}.

\section{Bowditch's theorem and its consequence}
In this section, we shall prove Proposition \ref{only identity}
by using the following acylindricity of the mapping class group action on the curve complex proved by Bowditch \cite{Bo}.
\begin{otheorem}[Bowditch \cite{Bo}]
\label{acylindrical}
For any given $D>0$, there are $R>0$ and 
a positive integer $N$
depending only on $D$ and the topological type of $S$ with the following property.
Let $x, y$ be two points in 
$\cc(S)$
with $d(x,y) \geq R$.
Then there are at most $N$ elements $g$ of $\mcg(S)$ such that both $d(x, g x) \leq D$ and $d(y, g y) \leq D$.
\end{otheorem}

Before starting the proof of Corollary \ref{only identity},
we prepare the following lemma.

\begin{lemma}
\label{qc}
There exists a constant 
$A>0$
depending only on $Q$ and $\delta$ for which the following holds.
Let $Y$ be a $Q$-quasi-convex set in a $\delta$-hyperbolic geodesic space $X$ and $\pi : X \rightarrow Y$ the nearest point retraction.
Then for any 
$B>0$,
there exists 
$C>0$,
depending only on $Q$, $\delta$ and $B$,
such that for any points $z, w \in X$ with $d(z,w) \leq B$ and $d(z, Y)\geq C, d(w, Y) \geq C$, we have $d(\pi(z), \pi(w)) \leq A$.
\end{lemma}
\begin{proof}
This is a consequence of the $\delta$-thinness of triangles in $X$.
We let $A$ be $2Q+3\delta+1$, and for given $B$, we set $C=B+Q+2\delta+1$.
Let $z,w$ be points in $X$ with $d(z,w) \leq B$ and $d(z, Y)\geq C, d(w, Y) \geq C$.
For two points in $X$, we shall denote a geodesic segment connecting them (which we choose) by putting bar over them.
We have only to show that $\length(\overline{\pi(z) \pi(w)}) \leq A$.

Suppose not.
Then there is a point $p \in \overline{\pi(z) \pi(w)}$ with $d(p, \pi(z)) \geq  Q+\delta+1/2$ and $d(p, \pi(w)) \geq Q+2\delta+1/2$.
We consider two geodesic triangles $\triangle_w=\overline{zw} \cup \overline{w \pi(w)} \cup \overline{z \pi(w)}$ and $\triangle_z=\overline{z \pi(z)} \cup \overline{\pi(z)\pi(w)} \cup \overline{z \pi(w)}$.
By the $\delta$-thinness of $\triangle_z$, there is 
a point $q$ on either  $\overline{z\pi(z)}$ or $\overline{z \pi(w)}$ with $\maptripod_{\triangle_z}(q)=\maptripod_{\triangle_z}(p)$ and $d(p,q) \leq \delta$.
Suppose that  $q$ lies on $\overline{z \pi(z)}$.
Then $d(q, \pi(z))=d(q, Y) \leq d(q,p)+d(p, Y) \leq \delta+Q$.
On the other hand, since $\maptripod_{\triangle_z}(q)=\maptripod_{\triangle_z}(p)$, we have $d(q,\pi(z))=d(p, \pi(z))\geq Q+\delta+1/2$.
This is a contradiction.
This contradicts the assumption 
Thus we see that $q$ lies on $\overline{z \pi(w)}$.

Now we turn to consider the triangle $\triangle_w$.
There is a point $r$ on either $\overline{zw}$ or $\overline{w\pi(w)}$ such that $\maptripod_{\Delta_w}(q)=\maptripod_{\Delta_w}(r)$ and hence $d(q,r) \leq \delta$ by the $\delta$-thinness of $\triangle_w$.
Suppose first that $r$ lies on $\overline{zw}$.
Then $d(r, w) \leq d(z,w) \leq B$.
Therefore, we have $d(w, Y) \leq d(w,r)+d(r,q)+d(q,p) +d(p,Y)\leq B+2\delta+Q$.
This contradicts the assumption that
$d(w,Y) \geq C=B+Q+2\delta+1$.

Suppose next that $r$ lies on $\overline{w \pi(w)}$.
Then $d(r,\pi(w))=d(r,Y) \leq d(r,q)+d(q,p)+d(p,Y) \leq 2\delta+Q$.
On the other hand, the equalities $\maptripod_{\triangle_w}(r)=\maptripod_{\triangle_w}(q), \maptripod_{\triangle_z}(q)=\maptripod_{\triangle_z}(p)$ imply $d(r, \pi(w))=d(q,\pi(w))=d(p,\pi(w))\geq Q+2\delta+1/2$.
This is a contradiction.
Thus we have proved that $d(\pi(z), \pi(w)) \leq A$.
\end{proof}

\begin{proof}[Proof of Proposition \ref{only identity}]
Recall that $\Delta_i\, (i=1,2)$ is $L$-quasi-convex.
This implies that for $r \geq L$, the $r$-neighbourhood, $N_r(\Delta_i)$, of $\Delta_i$ is 
$2\delta$-quasi-convex.
(See \cite[Chap 10, Proposition 1.2]{CDP}.)
Let $\pi_r:\cc(S)\to N_r(\Delta_i)$ be a nearest point projection.
Letting $Q$ in 
Lemma \ref{qc} 
be $2\delta$, 
we get a constant $A>0$,
which satisfies the following condition:
For any $B>0$, there exists $C>0$, 
such that for any points $z, w \in \cc(S)$ with $d(z,w) \leq B$ and 
$d(z, N_r(\Delta_i))\geq C, d(w, N_r(\Delta_i)) \geq C$,
we have $d(\pi_r(z), \pi_r(w)) \leq A$ for every $r\geq L$.

Regarding this $A$ as $D$ in Theorem \ref{acylindrical}, 
we get a constant $R$ and a positive integer $N$,
which satisfy the following condition:
For two points $x, y$  in $\cc(S)$  with $d(x,y) \geq R$,
there are at most $N$ elements $g$ of $\mcg(S)$ 
such that both $d(x, g x) \leq A$ and $d(y, g y) \leq A$.

By setting $B=(N+1)E$ for a given $E$ in our statement in 
Proposition \ref{only identity}
and recalling the choice of the constant $A$ using Lemma \ref{qc},
we obtain a constant $C$
which satisfies the following condition:
For any points $z, w \in \cc(S)$ with $d(z,w) \leq B$ and 
$d(z, N_r(\Delta_i))\geq C, d(w, N_r(\Delta_i)) \geq C$,
we have $d(\pi_r(z), \pi_r(w)) \leq A$.
Here $r$ is any real number with $r\geq L$.

After this preparation, we now prove the conclusion of our proposition holds
if we let $L_0$ be $L+C+R$.
Suppose on the contrary 
that there is a subsegment $\gamma'$ 
of $\gamma$
of length $L_0$ on which 
a non-trivial element $g\in G_i$ translates points within the distance $E$.
Let $\xi$ be the point on $\gamma'$ farthest from $\Delta_i$.
Then $d(\xi, \Delta_i)\ge L_0=L+C+R$.
Thus, for each $r$ with $L\leq r\leq L+R$,
we have $d(\xi, N_r(\Delta_i))\ge C$.
On the other hand, for each $j$ with $1\le j\le N+1$,
we have $d(\xi,g^j(\xi))\leq jE\leq (N+1)E=B$.
Hence $d(\pi_r(\xi),\pi_r(g^j(\xi)))\le A$
for each $r$ with $L\leq r\leq L+R$.
Note that we may choose $\pi_r$ so that $\pi_r(g^j(\xi))=g^j(\pi_r(\xi))$.
Thus we have $d(\pi_r(\xi),g^j(\pi_r(\xi)))\le A$
for each $r$ with $L\leq r\leq L+R$
and for each $j$ with $1\le j\le N+1$.
Since we can assume that $\pi_r(\xi)$ lies on $\gamma$, this shows that there is a subsegment $\gamma''$ of $\gamma$ with length $R$ on which all points are translated within the distance $A$ by any $g^j$ for $j=1, \dots N+1$.
Since any element in $G_i$ has infinite order as was shown in the proof of Proposition 1.7 in Otal \cite{Ot}, we see that $g, \dots , g^{N+1}$ are all distinct.
This contradicts Theorem \ref{acylindrical}.
\end{proof}

\section{Non-trivial curves}
In this section and the next, we only consider Heegaard splittings, for our argument relies on the construction of  model manifolds for Heegaard splittings 
due to Namazi \cite{Na, Na1} and 
its generalisation by 
Namazi-Souto \cite{NS} and Brock-Minsky-Namazi-Souto \cite{BMNS}.
We believe that we can obtain the same result for bridge decompositions since their theory is valid in more general settings including the case where the hyperbolic manifolds have torus cusps as is suggested in  \cite{BMNS}.

Before stating the theorem, we shall review the definitions of subsurface projections due to Masur-Minsky \cite{MaMi2} and of bounded combinatorics introduced by Namazi \cite{Na}.
Let $S$ be a closed surface and 
$Y$ a connected open incompressible subsurface of $S$,
which is either an annulus or has negative Euler characteristic.
The curve complex $\cc(Y)$ is defined in the same way as $\cc(S)$ unless $Y$ is either a once-punctured sphere or four-times punctured sphere or an annulus.
When $Y$ is either a once-punctured torus or a four-times punctured sphere, $\cc(Y)$ is a one-dimensional simplicial complex whose vertices are the isotopy classes of essential simple closed curves and where two vertices are connected if their geometric intersection number is the least possible: $1$ when $Y$ is a once-punctured torus and $2$ when $Y$ is a four-times punctured sphere.
When $Y$ is an annulus, we consider its compactification $\bar Y$, and $\cc(Y)$ is defined to be a one-dimensional simplicial complex whose vertices are the isotopy classes relative to the endpoints of essential arcs and where two vertices are connected if they are realised to be disjoint.

Let $c$ be an essential simple closed curve on $S$, and isotope it so that $\Fr Y \cap c$ is transverse and there are no inessential intersection in $Y \cap c$.
If $Y$ has negative Euler characteristic, 
we consider, roughly speaking, all possible essential simple closed curves obtained by connecting endpoints of $F \cap c$ by arcs on $\Fr Y$, and denote the set consisting of such simple closed curves by $\pi_Y(c)$ regarding it as a subset of $\cc(Y)$.
If there are no such simple closed curves, we define $\pi_Y(c)$ to be the empty set.
When $Y$ is an annulus, $\pi_Y(c)$ 
is defined to be 
the set of
points in $\cc(Y)$ determined by the lifts of $c$ to the covering of $S$ associated to $\pi_1(Y)$ and its natural compactification as a hyperbolic surface.
For a subset $C$ of $\cc(S)$, we define $\pi_Y(C)$ to be $\cup_{c \in C} \pi_Y(c)$.
Also, we can define the projections of multi-curves and clean markings in the same way.
See \cite[Section 2]{MaMi2} for precise definition.

\begin{Def}
\label{def_bounded-combinatorics}
Let $M=H_1 \cup H_2$ be a Heegaard splitting along $S$ 
and let $\Delta_1$ and $\Delta_2$ be the subsets of $\cc(S)$
consisting of the meridians of $H_1$ and $H_2$, respectively.
For a positive real number $R$,
we say that the decomposition has 
{\it $R$-bounded combinatorics} if 
there are handlebody pants decompositions $P_1\subset \Delta_1$ and 
$P_2\subset \Delta_2$ of $H_1$ and $H_2$,
such that the distance 
between
 $P_1$ and $P_2$ in $\cc(S)$ is equal to the distance between $\Delta_1$ and $\Delta_2$,
and that
the distance in $\cc(Y)$ between $\pi_Y(P_1)$ and $\pi_Y(P_2)$ is bounded by $R$ for any proper 
incompressible
subsurface $Y$ of $S$
which is either an annulus or has negative Euler characteristic.
\end{Def}

\begin{theorem}
\label{main1}
For any given positive constant $R$, there is a constant 
$K_0$
depending only on $R$ and the 
topological type
of $S$ with the following property.
If a closed hyperbolic $3$-manifold $M$ has a Heegaard 
splitting
$M=H_1\cup_S H_2$ with $R$-bounded combinatorics whose Hempel distance  is greater than or equal to $K_0$,
then there is 
a non-empty open set
$O$ in $\pl(S)$ such that 
no simple 
closed curves
in $S$
representing a point in $O$ 
is null-homotopic in $M$.
Furthermore, two simple 
closed curves
in $S$
representing distinct points 
in $O$
cannot be homotopic in $M$.
\end{theorem}
In the following argument, we fix a homeomorphism type of $S$ once and for all.

As is noted in the above,
our proof of this theorem relies on the work of Namazi \cite{Na, Na1} and its generalisation by Brock-Minsky-Namazi-Souto \cite{BMNS} on model manifolds of Heegaard splittings and more complicated glueing.
They showed that for a given positive constant $R$, if we take a sufficiently large $K_0$, 
then
for any   $3$-manifold $M$ which admits a Heegaard splitting
$M=H_1 \cup_S H_2$ with $R$-bounded combinatorics and Hempel distance 
$\ge K_0$,
there is a bi-Lipschitz model manifold which is a pinched negatively curved $3$-manifold 
homeomorphic to $M$ with its pinching constants depending only 
on $K_0$ and the topological type of $S$,
and is obtained by pasting hyperbolic handlebodies as will be explained below.

The model manifold has a negatively curved metric obtained by  glueing hyperbolic metrics on $H_1$ and $H_2$ using a hyperbolic 3-manifold homeomorphic to $S \times \reals$.
We shall explain how to do this following the description by Namazi \cite{Na}.
We first choose handlebody pants decompositions $P_1\subset \Delta_1$ and 
$P_2\subset \Delta_2$ of $H_1$ and $H_2$ as in Definition \ref{def_bounded-combinatorics}, which realises the distance between $\Delta_1$ and $\Delta_2$.
By \cite[Lemma 2.9]{Na}, 
there are clean markings $\alpha_1, \alpha_2$ with base curves $P_1, P_2$ respectively such that $\pi_Y(\alpha_1)$ and $\pi_Y(\alpha_2)$ are within the distance $R$ in $\cc(Y)$ for any proper open incompressible subsurface $Y$ of $S$.
We take points $m_1, m_2$ in the Teichm\"uller space of $S$ such that 
the total length of $\alpha_j$ with respect to the hyperbolic metric compatible with $m_j$ 
is shortest among all clean markings.
Let $g_1$ and $g_2$ be convex cocompact hyperbolic metrics on $\Int H_1$ and $\Int H_2$  whose marked conformal structures at infinity are $m_2$ and $m_1$ respectively, where we define the markings of $H_1$ and $H_2$ by  regarding them as being embedded in $M$ with the common boundary $S$.
Then if we take $K_0$ to be sufficiently large, there are a doubly degenerate hyperbolic $3$-manifold $M_S$ homeomorphic to 
$S \times \reals$ 
and open sets 
$U_1 \subset \Int H_1, U_2 \subset \Int H_2$, 
$V_1, V_2 \subset M_S$ with the following conditions.

\begin{enumerate}[(a)]
\item The hyperbolic $3$-manifold $M_S$ has $\epsilon_0$-bounded geometry with $\epsilon_0$ depending only on $R$; that is, every closed geodesic in $M_S$ has length greater than or equal to $\epsilon_0$.
\item All of $U_1, U_2, V_1, V_2$ are homeomorphic to $S \times (0,1)$.
\item There is a bi-Lipschitz diffeomorphism between 
the subspace $U_j$ of the hyperbolic manifold $(\Int H_j, g_j)$ and 
the subspace $V_j$ of the hyperbolic manifold $M_S$
for $j=1,2$ which tends to an isometry uniformly in the $C^2$-topology as $K_0 \rightarrow \infty$.
\item The intersection $V_1 \cap V_2$ is also homeomorphic to $S \times (0,1)$.
\item We define the width
of $V_1 \cap V_2$ to be the distance between its two frontier components with respect to the metric of $M_S$.
Then, the width of $V_1 \cap V_2$ goes to $\infty$ as $K_0 \rightarrow \infty$.
\end{enumerate}

These hold because of the following facts.
Suppose that we are given a sequence of Heegaard splittings
as above with Hempel distance going to $\infty$.
We use the superscript $i$ to denote the $i$-th pants decompositions and metrics, 
as $P_j^i$ or $m_j^i$ or $g_j^i$ for $j=1,2$.
 If we fix a marking on $H_1$, then the convex cocompact hyperbolic structure $(\Int H_1, g_1^i)$ converges to a geometrically infinite hyperbolic structure $g_1^\infty$ in $\Int H_1$ both algebraically and geometrically, 
after passing to a subsequence.
The convergence is guaranteed by the facts that the total length of $P_2^i$,
with respect to the hyperbolic metric $m_2^i$ on $S$ corresponding to the conformal structure
at infinity associated with $g_1^i$,
is uniformly bounded 
and that $\{P_2^i\}$ converges to a lamination in the Masur domain of the projective lamination space passing to a subsequence.
The hyperbolic $3$-manifold $(\Int H_1, g_1^\infty)$ is asymptotically isometric to a doubly degenerate hyperbolic $3$-manifold $M_S$ whose ending laminations are 
the
limits of $\{P_1^i\}$ and $\{P_2^i\}$ respectively; 
that is, for any $\delta >0$, there exists a compact set $K$ such that $\Int H_1 \setminus K$ with the metric $g^\infty_1$  is embedded into a neighbourhood of 
the end of $M_S$, with ending lamination equal to the limit of $\{P_2^i\}$,
by a diffeomorphism which is $\delta$-close to an isometry in the $C^2$-topology.
Combined with the convergence which we have just explained, it follows that for any $\delta >0$, there exist 
a compact set $K$ 
and  a sequence $\{R_i\}$ going to $\infty$ such that for the $R_i$-neighbourhood $N_{R_i}(K)$ of $K$ 
in $\Int H_1$
with respect to the metric $g^i_1$, its subset $N_{R_i}(K) \setminus K$ is embedded into $M_S$ by a diffeomorphism $\delta$-close to an isometry. 
In the same way,  $(\Int H_2, g_2^\infty)$ is asymptotically isometric to $M_S$, and the complement of a compact set in $(\Int H_2, g_2^\infty)$ is embedded nearly isometrically to a neighbourhood of the other end of $M_S$.
These facts guarantee the existence of the  sets
$U_1 \subset \Int H_1, U_2 \subset \Int H_2$, 
$V_1, V_2 \subset M_S$ which appeared above.
(As for more details on the proof of this convergence and the property of its limit, refer to \cite{Oh2} and \cite{NS2}.)
We also note that the bi-Lipschitz diffeomorphisms in (c) are liftable as can be seen by checking the conditions in \cite[Lemma 3.1]{MiJ}.

Now the model manifold is constructed as follows.
We paste $(\Int H_1, g_1)$ and $(\Int H_2, g_2)$ to get a manifold homeomorphic to $M$, identifying $U_1$ with $V_1$ and $U_2$ with $V_2$ by bi-Lipschitz diffeomorphisms close to isometries defined above, letting $V_1 \cap V_2$ be the margin of pasting and throwing away neighbourhoods of ends outside $U_1$ and $U_2$.
Construct a Riemannian metric on the resulting manifold 
by glueing the hyperbolic metrics $g_1$ and $g_2$ 
along the margin of glueing by using a bump function.
Then, since both $g_1$ and $g_2$ get closer and closer 
to the hyperbolic metric of $M_S$ on a neighbourhood of the margin of glueing,
the sectional curvature of the resulting metric 
lies between $(-1-\varepsilon, -1+\varepsilon)$, with $\varepsilon \rightarrow 0$ as 
$K_0 \rightarrow \infty$.
Tian's theorem
\cite{Tian}
(cf.\ \cite[Chapter 12]{Na})
implies that this metric is $\rho$-close (as Riemannian metric) to the original hyperbolic metric of $M$ in the $C^2$-topology, 
where the constant $\rho$ depends only on $K_0$ and goes to $0$ as $K_0 \rightarrow \infty$.

In the following, we show that if we take the lower bound $K_0$ of the Hempel distance
to be large enough, 
then there is an open set $O$ in $\pl(S)$ satisfying the 
conditions in Theorem \ref{main1}.
As was explained above, the hyperbolic metric on $M$ is uniformly close to the constructed negatively curved model metric in the $C^2$-topology.
On the other hand, as was shown in the construction, the negatively curved metric on 
$M$ is close to the
hyperbolic metric on $M_S$ in the part corresponding to $V_1 \cap V_2$ which appeared above.
Recall that, for any positive constant $\epsilon$,
there is an upper bound  (depending only on $\epsilon$ and the topological type of $S$) for the diameters modulo their 
$\epsilon$-thin parts of the pleated surfaces
in $M_S$ intersecting $V_1 \cap V_2$ 
(see \cite[Chapter 9]{ThL} and \cite[Lemma 1.2]{Oh}).
Since $M_S$ has $\epsilon_0$-bounded geometry with $\epsilon_0$ depending only on $R$  in our case,
the injectivity radii on pleated surfaces are bounded from below by $\epsilon_0/2$.
Therefore, there is a constant $K$,
depending only on  $R$ and the topological type of $S$,
such that a pleated surface 
in $M_S$
which has a point in $V_1 \cap V_2$ at the distance at least $K$ from the frontier of $V_1 \cap V_2$ must be entirely contained in $V_1 \cap V_2$.

Now, recall that Thurston defined a notion of rational depth of measured laminations
(see \cite[Definition 9.5.10]{ThL}).
A measured lamination is said to have rational depth $k$ when it is carried by a train track with a weight system $w$ which has $k$ independent linear relations over $\rationals$ in addition to those coming from the switch conditions.
Thurston proved the following two facts:
(1) The set of measured laminations of rational depth $0$ has full measure in $\ml(S)$.
 (2) 
 For any
 two measured laminations, we can find 
an embedded arc 
in $\ml(S)$
 connecting them whose interior  passes only measured laminations of rational depth less than $2$, and only countably 
 many measured laminations of rational depth $1$.

One more thing proved by Thurston which we need to use now is  the existence of   one-parameter family consisting of pleated surfaces and negatively curved interpolated 
surfaces
between two pleated surfaces
(see \cite[Section 9.5]{ThL} and \cite[Section 4.E]{Oh}).
This shows that for any point in $M_S$, there is 
a negatively curved surface in the family passing through that point and that such a surface is contained in a uniformly bounded neighbourhood of a pleated surface realising a measured lamination of rational depth $0$.
Therefore, any point in $M_S$ has a pleated surface homotopic to the inclusion of $S$, which realises a measured lamination of rational depth $0$, 
which contains a point within a uniformly bounded distance from the give point.
In particular,
if the width of $V_1 \cap V_2$ is large enough, then there is a pleated surface $f$ realising a  measured lamination $\lambda$ of rational depth-$0$ whose image is contained in $V_1 \cap V_2$ and is away from the frontier of $V_1 \cap V_2$ by the distance greater than $L$ for any $L$ sufficiently smaller than the width of $V_1 \cap V_2$.
We shall specify this constant $L$ later.
For the moment, we just note that we can take $L$ to be large if $K_0$ is large.
Since we can take $V_1$ and $V_2$ so that the width of $V_1 \cap V_2$ goes to $\infty$ as the lower bound of Hempel distance $K_0 \rightarrow \infty$,
we may assume, by taking $K_0$ large enough,
that such a pleated surface $f$ realising a measured lamination $\lambda$ of rational depth-$0$ actually exists.

We shall show by contradiction that there is an open neighbourhood $O$ of $[\lambda]$ in $\pl(S)$ such that every lamination contained in $O$ is realisable by a pleated surface in $V_1 \cap V_2$ whose image is at the distance greater than $L$ from the frontier of $V_1 \cap V_2$.
Suppose that such an open set does not exist.
Since the realisability of a measured lamination is invariant under scalar multiplications, there is a sequence of measured laminations $\{\lambda_i\}$ converging to $\lambda$, which cannot be realised by pleated surfaces contained in $V_1 \cap V_2$
at the distance greater than $L$ from the frontier of $V_1\cap V_2$.
Since there are only two unmeasured laminations which are unrealisable in $M_S$ and the measured laminations having these laminations as supports constitute closed subsets of $\ml(S)$ not containing $\lambda$,  by taking a subsequence, we can assume that all the $\lambda_i$ can be realised by pleated surfaces in $M_S$.
Let $f_i: S \rightarrow M_S$ be a pleated surface realising $\lambda_i$.
Since $\{\lambda_i\}$ converges to $\lambda$, which does not represent an ending lamination, we see that there is a compact set which all the images of 
$f_i$ intersect.
Therefore, by passing to a subsequence,
$\{f_i\}$ converges uniformly to a pleated surface $f_\infty$ realising $\lambda$
(cf. \cite[Theorem 5.2.18]{CEG}).

Since $\lambda$ has rational depth $0$, its realisation is unique; hence $f_\infty$ coincides with the pleated surface $f$ which appeared above.
It follows that the image of $f_i$ is also contained in $V_1 \cap V_2$ and at the distance greater than $L$ from the frontier of $V_1\cap V_2$, for sufficiently large $i$.
This is a contradiction.
Thus we have shown that there is an open set $O$ containing 
$[\lambda]$
and consisting
of measured laminations which can be realised by pleated surfaces in $V_1 \cap V_2$ whose images are at the distance greater than $L$ from the frontier of $V_1 \cap V_2$.

Next we shall show that no curve contained in $O$ can be null-homotopic in $M$ if we take 
the constant $K_0$  (and hence also $L$) to be large enough
and make $O$ smaller accordingly.
Suppose that a simple closed curve  
$c$ on $S$ is contained in $O$.
Since $c$ is realised 
by a pleated surface contained 
in $V_1 \cap V_2$, there is a closed geodesic $c^*$ in $V_1 \cap V_2$ homotopic to $c$.
Recall that there is a $C$-bi-Lipschitz diffeomorphism from an open set $V'$ in $M$ to $V_1 \cap V_2$, which lifts to a $C$-bi-Lipschitz diffeomorphism from 
the universal cover
$\tilde V'$ of $V'$ to $\hyperbolic^3$ (with respect to the path metrics), with a constant $C$ depending only on $S$ and $K_0$ and going to $1$ as $K_0 \rightarrow \infty$.
Regard $c^*$ as a closed geodesic arc, and let $\tilde c$ be a geodesic arc
in the universal cover $\widetilde{V_1 \cap V_2} \subset \tilde M_S=\hyperbolic^3$
obtained as a lift of the geodesic arc $c^*$.
Then $\tilde c$ is pulled back to a 
$(C^2,0)-$quasi-geodesic
$\tilde c'$ in the universal cover $\tilde V'$ of $V'$ (see \cite[Lemma 2.36]{OhD}).
By the stability of quasi-geodesics (see  \cite[Chaptire 3, Th\'eor\`eme 1.2] {CDP} or \cite[Theorem 2.31]{OhD}), there is a constant $N$ depending only on $C$, which goes to $0$ as $C \rightarrow 1$ such that $\tilde c'$
can be homotoped to a geodesic arc fixing the endpoints within the $N$-neighbourhood of
$\tilde c'$ if $\tilde c'$
is at the distance greater than $N$ from the ends of $\tilde V'$.

To be more precise, we cannot directly use the stability which is valid only in a Gromov hyperbolic space, 
whereas 
our space $\tilde V'$
is only a locally Gromov hyperbolic space, $\tilde V'$ is provided with a path metric induced
from the Riemannian metric of  $\tilde V'$.
We take the metric completion of $\tilde V'$ and get a geodesic space $\overline V'$.
The triangles in our space $\overline V'$ 
are known to be thin only when they do not touch the boundary.
Still, we can apply the same argument as the proof of the stability for $\overline V'$,  and can show that the geodesic connecting the endpoint of $\tilde c'$ 
cannot touch the boundary since $\tilde c'$ is at the distance more than $C^{-1}L$ from the boundary, and also that the stability holds.

Thus if $\tilde c'$
is at the distance greater than $N$ from the ends of $\tilde V'$,
then $\tilde c'$
can be homotoped to a geodesic arc fixing the endpoints within the $N$-neighbourhood of
$\tilde c'$, and projecting this to $V'$, 
we obtain a closed geodesic 
arc (not necessarily a closed geodesic) homotopic to
the pull-back of $c^*$ in $M$.
Therefore, if we take $K_0$ and  $L$ large enough so that the frontier of $V_1 \cap V_2$ is at the distance greater than $CN$, we see that the pull back of $c^*$ can be homotoped to a closed geodesic 
arc hence cannot be null-homotopic.

Thus we have shown that no 
simple closed curve
in $O$ is null-homotopic in $M$ if $K_0$ is large enough, and completed the proof of the first part of 
Theorem \ref{main1}.

It remains to show that two 
simple closed curves
in $O$ which are not homotopic on $S$ are not homotopic in $M$.
Suppose that two simple 
closed curves
$c_1$ and $c_2$ in $S$ 
represent distinct points in $O$.
As in the previous paragraph, we have quasi-geodesics $c_1', c_2'$ homotopic to $c_1, c_2$ in $M$, which are obtained by pulling back the closed geodesics $c_1^*$ and $c_2^*$  in $V_1 \cap V_2$ representing the free homotopy classes of $c_1$ and $c_2$.
We consider their infinite lifts $\hat c_1', \hat c_2'$  in $\tilde V'$, which are $(C^2,0)$-quasi-geodesic lines.
By applying the same argument as above using 
the stability of quasi-geodesic lines instead of that of quasi-geodesic arcs, we see that $\hat c_1', \hat c_2'$ can be homotoped to geodesic lines $\hat c_1^*, \hat c_2^*$ in $\tilde V'$ by a proper homotopy which doe not touch the boundary of $\tilde V'$.
By projecting them down to $M$, we get closed geodesics ${c_1'}^*, {c_2'}^*$, which must coincide since we assumed that $c_1$ and $c_2$ are homotopic in $M$.
It follows  that $c_1'$ and $c_2'$ are homotopic in $V'$, which implies that 
 $c_1^*$ and $c_2^*$ are homotopic in 
$V_1\cap V_2\cong S\times (0,1)$.
This contradicts our assumption that $c_1$ and $c_2$ are not homotopic in $S$.
This shows that $c_1$ and $c_2$ cannot be homotopic in $M$, and we have completed the proof of Theorem \ref{main1}.
\vskip 0.5cm

Next we shall show that $\langle G_1, G_2 \rangle$ has a non-empty domain of discontinuity on $\pl(S)$.

\begin{theorem}
\label{discontinuous open set}
Let $O$ be an open set in $\pl(S)$ as in Theorem \ref{main1}.
Then we have $\{g\in \langle G_1, G_2 \rangle \mid g O \cap O \neq \emptyset\}=\{1\}$.
\end{theorem}

\begin{proof}
Suppose that 
$U:=gO \cap O$
is non-empty.
Let $c$ be any simple closed curve 
in $U$.
Then, there exists a simple closed curve $c'$ in $O$
with $g(c')=c$.
Since $\langle G_1, G_2 \rangle$ acts on $\pi_1(M)$ trivially, we see that $c$ is freely homotopic to $c'$ in $M$.
Since 
$c$ and $c'$ are
contained in $O$, by Theorem \ref{main1}, it follows that $c=c'$.
This shows that $g$ fixes every simple closed curve contained in $U$.
Since the simple closed curves are dense in $U$, this implies that $g$ fixes $U$ pointwise.
It is easy to see as follows that such an element in $\mcg(S)$ must be the identity.
Since $g$ is not a torsion, it is either pseudo-Anosov or reducible.
If $g$ is pseudo-Anosov, it fixes only two points in 
$\pl(S)$.
If $g$ is reducible, some of its powers, $g^p$ is 
either partially pseudo-Anosov, \ie, there is a subsurface $T$ of $S$ such that $g^p|T$ is pseudo-Anosov, or a product of nontrivial powers of Dehn twists along mutually disjoint 
essential simple closed curves $c_1,\cdots, c_m$ on $S$.
In the first case, 
let $\lambda_u, \lambda_s$ be the unstable and stable laminations of $g^p|T$.
Then $g^p$ fixes only a measured lamination whose supports contain either $|\lambda_u|$ or $|\lambda_s|$.
Such laminations cannot constitute an open set.
Therefore $g^p$ cannot fix all points in $U$.
In the latter case, the fixed point set of the action of $g^p$ on
the projective measured lamination space consists of those elements
whose underlying geodesic laminations do not intersect $\cup c_i$ transversely.
Such a subspace of 
$\pl(S)$ 
cannot contain 
a nonempty open set.
Therefore $g^p$ cannot fix all points in $U$ in this case, too.
\end{proof}

\section{Iteration of pseudo-Anosov map}

In the case when two handlebodies are pasted by an $n$-time iteration of a pseudo-Anosov map, we can have a region as in Theorems \ref{main1} and \ref{discontinuous open set}, which gets larger and larger as $n \rightarrow \infty$ to cover the complement of  the closure of meridians and 
isolated points as we shall see below.

Let $H_1$ and $H_2$ be handlebodies whose boundary is identified with a closed
orientable surface $S$ of genus $>1$.
Let $\phi: S \rightarrow S$ be a pseudo-Anosov map with a stable lamination $\mu_\phi$ and an unstable lamination $\lambda_\phi$.
For each positive integer $n$, we consider the $3$-manifold 
$M_n=H_1 \cup_{\phi^n} H_2$
obtained by identifying the boundaries of $H_1$ and $H_2$ via $\phi^n:\partial H_1=S\to S=\partial H_2$.
As in the introduction,
let $G_1$ be the subgroup of $\aut(\cc(S))$ obtained as the image 
of the subgroup of $\mcg(S)$ consisting of those classes of 
auto-homeomorphisms of $H_1$  which are homotopic to the identity in $H_1$.
The other handlebody $H_2$ is regarded as embedded in $M_n$ as 
$H_2^{(n)}$ and $\partial H_2^{(n)} =\partial H_1=S$.
We let $G^{(n)}_2$ be the subgroup of 
the subgroup of $\aut(\cc(S))$ obtained as the image 
of the subgroup of $\mcg(S)$ consisting of those classes of 
auto-homeomorphisms of $H_2^{(n)}$  which are homotopic to the identity in $H_2^{(n)}$.
We denote by $G^{(n)}$  the subgroup of 
$\aut(\cc(S))$
generated by $G_1$ and $G_2^{(n)}$.

In the measured lamination space $\ml(\partial H_1)$, we define the set of doubly incompressible laminations $\mathcal D(H_1)$ to be
$$\mathcal D(H_1)=\{\lambda \in \ml(\partial H_1) \mid \exists \eta>0 \text{ such that } i(\lambda, m) >\eta \text{ for every meridian } m\},$$
following Lecuire \cite{Lec} and Kim-Lecuire-Ohshika \cite{KLO}.
We denote the projection of $\mathcal D(H_1)$ to $\pl(\partial H_1)$ by 
$\mathcal{PD}(H_1)$.
As was shown in Theorem 1.4 in Lecuire \cite{LecS}, the group $G_1$ acts on $\mathcal D(H_1)$ properly discontinuously.
A smaller open set called the Masur domain is defined as follows.
First we define 
$\mathcal C(H_1)$
by
\begin{equation*}
\begin{split}
\mathcal C(H_1)=\{c \in \ml(\partial H_1) \mid &\text{ the support of } c \text{ is a simple closed curve}\\ &\text{ bounding a disc in } H_1\}.
\end{split}
\end{equation*}
We then define the Masur domain to be
$$\mathcal M(H_1) = \{ \lambda \in \ml(S) \mid i(\lambda, \mu) >0 \text{ for any } \mu \in \overline{C(H_1)}\}.$$
We note that $\mathcal D(H_1)$ contains $\mathcal M(H_1)$ and that any arational lamination 
in $\mathcal D(H_1)$ is also contained in $\mathcal M(H_1)$
(see \cite[Lemmas 3.1 and 3.4]{Lec}).
Since $\mathcal{D}(H_1)$ contains the projectivised Masur domain of $H_1$
which was proved to have full measure in $\pl(S)$ by Kerckhoff \cite{Ker}, $\mathcal{PD}(H_1)$ also has full measure.

\begin{theorem}
\label{iteration}
Assume that $\lambda_\phi$ is contained in $\mathcal M(H_1)$ and 
$\mu_\phi$ is contained in $\mathcal M(H_2)$.
Then for any projective lamination $[\lambda]$ in 
$\mathcal{PD}(H_1) \setminus G_1[\lambda_\phi]$,
there are an open neighbourhood $U$ and $n_0 \in \naturals$ 
such that for any $n \geq n_0$, 
no simple closed curve whose projective class is contained in $U$ is null-homotopic, 
and $\{g \in G^{(n)} \mid g U \cap U \neq \emptyset\}$ is finite.
\end{theorem}

\begin{proof}
Let $\iota_n : H_1 \rightarrow M_n$ be the inclusion.
Let $\phi_n : \pi_1(M_n) \rightarrow \psl$ be a representation corresponding to the hyperbolic structure on $M_n$.
In \S 5 of Namazi-Souto \cite{NS} it was proved that  under our assumptions, 
$\{\phi_n \circ \iota_n\}$ converges up to conjugations,
where we continue to denote by $\iota_n$ the homomorphism
$\pi_1(H_1)\to \pi_1(M_n)$ induced by $\iota_n$.
Note that we see by \cite[Lemma 3.4]{Lec}
that the homeomorphism $\phi$ is generic in the sense of 
\cite[Definition 2.1]{NS}
i.e., the unstable lamination $\lambda_{\phi}$ is not a limit, in $\pl(\partial H_1)$,
of meridians of $H_1$, and the stable lamination $\mu_{\phi}$ is not a limit of meridians of $H_2$.

Fix some nontrivial
element $\gamma \in \pi_1(H_1)$.
For any sufficiently large $n$, the element $\iota_n(\gamma)$ represents a non-trivial element of $\pi_1(M_n)$ 
(see \cite[Lemma 6.1]{NS}).
Let $\gamma^*_n$ be a closed geodesic in $M_n$ representing $\gamma$ in $\pi_1(M_n)$.
Since $\phi_n \circ \iota_n$ converges up to conjugations, the length of $\gamma^*_n$ is bounded as $n \rightarrow \infty$ and $\gamma^*_n$ converges geometrically to the projection in the geometric limit  of some closed geodesic in the algebraic limit.
Hence the distance between $\gamma^*_n$ and a spine of $H_1$ is bounded.
Therefore, if we take a base point $x_n$ on $\gamma_n^*$, then the geometric limit of $(M_n, x_n)$ is a hyperbolic $3$-manifold whose fundamental group is identified with $\pi_1(H_1)$.
We denote this geometric limit by $M_\infty$.
Namazi-Souto showed that this $M_\infty$ coincides with the hyperbolic $3$-manifold corresponding to the limit $\phi_\infty$ of $\{\phi_n \circ \iota_n\}$
(see \cite[Theorem 5.2]{NS}).
This hyperbolic $3$-manifold has a unique end, which has ending lamination corresponding to the limit of meridians of $H_2$ regarded as a simple closed curve on $\partial H_1$ by pasting map $\phi^{-n}$ as $n \rightarrow \infty$, which coincides with $\lambda_\phi$.
(The limit does not depend on the choice of meridian.) 
By the uniqueness of ending lamination and Theorem 5.1 in Lecuire \cite{Lec}, every lamination contained in 
$\mathcal{PD}(H_1) \setminus G_1[\lambda_\phi]$
is realisable by a pleated surface homotopic to the inclusion of $\partial H_1$.

Now, let $[\lambda]$ be a projective lamination  contained in 
$\mathcal{PD}(H_1) \setminus G_1[\lambda_\phi]$
Since $G_1$ acts on $\mathcal{PD}(H_1)$ properly discontinuously,
$\mathcal{PD}(H_1) \setminus G_1[\lambda_\phi]$
is an open set,
and hence
there is a neighbourhood $U$ of $[\lambda]$ which is contained in 
$\mathcal{PD}(H_1) \setminus G_1[\lambda_\phi]$.
We take $U$ so that its closure $\bar U$ is compact and is also contained in 
$\mathcal{PD}(H_1) \setminus G_1[\lambda_\phi]$.
Then every simple closed curve in $\bar U$ is realised by a pleated surface in $M_\infty$.
We shall show that there is $n_0 \in \naturals$ such that for every $n \geq n_0$ and every simple closed curve in $U$ represents a non-trivial class in $\pi_1(M_n)$.

Suppose not.
Then there is a sequence of simple closed curves with $[\gamma_n]$ contained in $U$ such that $\gamma_n$ is null-homotopic in $M_n$.
Taking a subsequence, $[\gamma_n]$ converges to a projective lamination $[\mu]$ passing to a subsequence.
The projective lamination $[\mu]$ is contained in $\bar U$; hence by our choice of $U$, it is contained in 
$\mathcal{PD}(H_1) \setminus G_1[\lambda_\phi]$.
Now, by the same argument as  \S 6 in Ohshika \cite{Oh}, this implies that there exists a neighbourhood $V$ of $[\mu]$ and $N \in \naturals$ such that every lamination in $V$ is realisable in $M_n$ for $n \geq N$.
Indeed, since $\mu$ is realisable in the geometric limit $M_\infty$, there is a train track $\tau$ with small curvature carrying $\mu$ such that every measured lamination carried by $\tau$ is realised in a fixed neighbourhood of the realisation of $\mu$.
By pulling back this to $M_n$ by an approximate isometry, we see that we have a neighbourhood as above.

Now, the existence of such $V$ contradicts the assumption that $[\gamma_n]$  is null-homotopic in $M_n$.
Thus we have shown that there is a neighbourhood $U$ of $\lambda$ such that for every $n \geq n_0$ and every simple closed curve in $U$ represents a non-trivial class in $\pi_1(M_n)$.
Repeating the argument in the proof of Theorem~\ref{main1}, we can also show that by letting $U$ smaller and $n_0$ larger if necessary, no two distinct simple closed curves in $U$ are homotopic in $M_n$ for $n \geq n_0$.

Moreover, by the same argument as in the proof of Theorem \ref{discontinuous open set}, we can show that by letting $U$ be smaller and $n_0$ larger, we have the condition  $\{g \in G^{(n)} \mid gU \cap U \neq \emptyset\}$ is finite for every $n \geq n_0$.
\end{proof}

\end{document}